\documentclass[11pt,dvipsnames,svgnames,table]{article}
\usepackage[T1]{fontenc}
\usepackage{graphicx,amsmath,amsthm,amsfonts,amssymb,microtype,thmtools,mathtools,anyfontsize,thm-restate,verbatim,tikz}
\usepackage{xcolor}
\usetikzlibrary{intersections}
\usetikzlibrary{external}
\usepackage{esvect}
\usepackage{wrapfig}
\usepackage{paralist, tabularx}
\usepackage{subcaption}
\usepackage{float}
\usepackage{bm}
\usepackage[shortlabels]{enumitem}
\setlist[itemize]{topsep=0ex,itemsep=0ex,parsep=0.4ex}
\setlist[enumerate]{topsep=0ex,itemsep=0ex,parsep=0.4ex}
\usepackage[unicode=true]{hyperref}
\hypersetup{
colorlinks=true,
breaklinks=true,
linkcolor={black},
citecolor={black},
urlcolor={blue!60!black},
pdftitle={Sparse String Graphs and Region Intersection Graphs over Minor-Closed Classes have Linear Expansion},
pdfauthor={Nikolai Karol and David R. Wood}}
\usepackage[noabbrev,capitalise]{cleveref}
\crefname{lem}{Lemma}{Lemmas}
\crefname{thm}{Theorem}{Theorems}
\crefname{cor}{Corollary}{Corollaries}
\crefname{prop}{Proposition}{Propositions}
\crefname{conj}{Conjecture}{Conjectures}
\crefname{open}{Open Problem}{Open Problems}
\crefname{question}{Question}{Questions}
\crefname{claim}{Claim}{Claims}
\crefformat{equation}{(#2#1#3)}
\Crefformat{equation}{Equation #2(#1)#3}

\newcommand{\defn}[1]{\textcolor{Maroon}{\emph{#1}}}

\newcommand{\GG}{\mathcal{G}}

\usepackage[longnamesfirst,numbers,sort&compress]
{natbib}
\makeatletter
\def\NAT@spacechar{~}
\makeatother
\usepackage[margin=30mm]{geometry}
\renewcommand{\geq}{\geqslant}
\renewcommand{\leq}{\leqslant}

\DeclareMathOperator{\dist}{dist}

\DeclareMathOperator{\scol}{scol}

\DeclareMathOperator{\sreach}{reach}
\DeclareMathOperator{\polylog}{polylog}
\renewcommand{\thefootnote}{\fnsymbol{footnote}}
\theoremstyle{plain}
\newtheorem{thm}[equation]{Theorem}
\newtheorem{lem}[equation]{Lemma}

\newtheorem{cor}[equation]{Corollary}
\newtheorem{prop}[equation]{Proposition}

\newtheorem{claim}{Claim}[thm]

\theoremstyle{definition}

\begin{document}

\author{
Nikolai Karol\footnotemark[2] \qquad
David~R.~Wood\footnotemark[2]}

\footnotetext[2]{School of Mathematics, Monash   University, Melbourne, Australia  (\texttt{\{Nikolai.Karol,David.Wood\}@monash.edu}).}

\sloppy

\title{\bf\boldmath 
Sparse String Graphs and Region Intersection Graphs over Minor-Closed Classes have Linear Expansion
}

\maketitle


\begin{abstract} 
We prove that sparse string graphs in a fixed surface have linear expansion. We extend this result to the more general setting of sparse region intersection graphs over any proper minor-closed class. The proofs are combinatorial and self-contained, and provide bounds that are within a constant factor of optimal. Applications of our results to graph colouring are presented.
\end{abstract}

\renewcommand{\thefootnote}
{\arabic{footnote}}

\section{Introduction} \label{section:intro}

A \defn{string graph} $G$ is the intersection graph of a collection $\mathcal{C}$ of non-self-intersecting continuous curves (also called \defn{strings}) in a surface. That is, the vertices of $G$ are represented by the curves of $\mathcal{C}$ such that two vertices of $G$ are adjacent if and only if the corresponding curves intersect.

String graphs developed out of the study of patterns of mutations in DNA sequences~\citep{Benzer59}, and of electrical networks realisable by printed circuits~\citep{Sinden66}. String graphs
were first formally defined by \citet{EET76} in 1976. Ever since, string graphs have been extensively studied; see \citep{Mat14,Pawlik14,SSS-JCSS03,Krat-JCTB91,Mat15,Lee17,SS-JCSS04,FP10,FP12,FP14,PRY20,Tom24,KM91,PT06,Kratochvil91,Davies25,CCTZ25,Karol25} for example.

This paper focuses on the following extension of string graphs, introduced by \citet{Lee17}. A graph $G$ is a \defn{region intersection graph} over a graph $H$ if there exists a collection $(A_{v} \subseteq H: v \in V(G))$ of connected subgraphs of $H$ such that $uv \in E(G)$ if and only if $V(A_u) \cap V(A_v) \neq \emptyset$. A \defn{class} is a collection of graphs, closed under isomorphism. A region intersection graph over a graph class $\mathcal{G}$ is a region intersection graph over a graph in $\mathcal{G}$.
It is straightforward to show that string graphs in the plane are precisely the region intersection graphs over planar graphs (see \citep[Lemma~1.4]{Lee17}). An analogous proof shows that for any surface $\Sigma$, string graphs in $\Sigma$ are precisely the region intersection graphs over the class of graphs embeddable in $\Sigma$.

It is widely recognised that string graphs constitute a complicated class of graphs. For example, recognition of string graphs is NP-hard~\citep{Krat-JCTB91,SSS-JCSS03}. There are triangle-free string graphs with arbitrarily large chromatic number~\citep{Pawlik14}. There exist string graphs requiring exponentially many intersection points in any realisation by a collection of curves~\citep{KM91}. The number of labelled $n$-vertex string graphs in the plane is $2^{(\frac{3}{4} + o(1))\binom{n}{2}}$~\citep{PT06}, while (for example) the number of labelled $n$-vertex graphs in any proper minor-closed\footnote{A graph $J$ is a \defn{minor} of a graph $G$ if a graph isomorphic to $J$ can be obtained from $G$ by vertex deletion, edge deletion, and edge contraction. A graph class $\mathcal{G}$ is \defn{minor-closed} if for every graph $G \in \mathcal{G}$, every minor of $G$ is in $\mathcal{G}$. A minor-closed class $\mathcal{G}$ is \defn{proper} if $\mathcal{G}$ is not the class of all graphs, so $\mathcal{G}$ must be $H$-minor-free for some graph $H$.} class is $2^{\mathcal{O}(n \log n)}$~\citep{NSTW06}.

The primary contribution of this paper is to show a new positive result about the structure of sparse string graphs, and to extend this to sparse region intersection graphs over proper minor-closed classes. Specifically, we show they have linear expansion, which is a strong property in the Graph Sparsity Theory of \citet{Sparsity}. To explain this result, several definitions are needed. For a graph $G$ and a set of vertices $S \subseteq V(G)$, the subgraph of $G$ \defn{induced by $S$}, denoted \defn{$G[S]$}, has vertex set $S$ and its edge set is the set of edges of $G$ with both endpoints in $S$.  A \defn{model} of a graph $J$ in $G$ is a collection of sets $(S_{i} : i \in V(J))$ indexed by the vertices of $J$ such that:

\begin{itemize}
    \item for each $i \in V(J)$, the set $S_{i} \subseteq V(G)$ is non-empty and $G[S_{i}]$ is connected,
    \item $S_i \cap S_j = \emptyset$ for all distinct $i, j \in V(J)$, and
    \item for every edge $ij \in E(J)$, there is an edge of $G$ between $S_i$ and $S_j$. 
\end{itemize}

The sets $S_{i}$ are called \defn{branch sets} of the model. It is folklore that $J$ is a minor of $G$ if and only if there is a model of $J$ in $G$.

The \defn{radius} of a connected graph $S$ is the minimum non-negative integer $r$ such that for some vertex $c \in V(S)$ and for every vertex $w \in V(S)$, we have $\dist_{S}(c, w) \leqslant r$. Such a vertex $c$ is called a \defn{centre} of $S$. For a graph $J$ and an integer $r \geqslant 0$, a model $(S_{i} : i \in V(J))$ of $J$ in a graph $G$ is \defn{$r$-shallow} if for each $i \in V(J)$, the radius of $G[S_{i}]$ is at most $r$. A graph $J$ is an \defn{$r$-shallow minor} of a graph $G$ if there is an $r$-shallow model of $J$ in $G$.

The \defn{edge density} of a graph $G$ is $|E(G)|/|V(G)|$ if $V(G)\neq\emptyset$, and is 0 if $V(G)=\emptyset$. For a graph $G$, let \defn{$\nabla_r(G)$} be the maximum edge density of an $r$-shallow minor of $G$. A graph class $\mathcal{G}$ has \defn{bounded expansion} if there is a function $f$ such that $\nabla_r(G)\leq f(r)$ for every graph $G\in\mathcal{G}$ and integer $r\geq 0$. Often the magnitude of such a function $f$ matters. A graph class $\mathcal{G}$ has \defn{polynomial expansion} if there exists  $c\in\mathbb{R}$ such that $\nabla_r(G)\leq c(r+1)^c$ for every graph $G\in\mathcal{G}$ and integer $r\geq 0$. As an illustrative example, the class of graphs with maximum degree at most 3 has bounded expansion (with expansion function $f(r)\in \mathcal{O}(2^r)$), but does not have polynomial expansion. A graph class $\mathcal{G}$ has \defn{linear expansion} if there exists  $c\in\mathbb{R}$ such that $\nabla_r(G)\leq c(r+1)$ for every graph $G\in\mathcal{G}$ and integer $r\geq 0$. As another illustrative example, the class of 3-dimensional grid graphs has polynomial expansion (with expansion function $f(r)\in \mathcal{O}(r^2)$), but does not have linear expansion.

The \defn{maximum density} of a graph $G$ is the maximum edge density of a subgraph of $G$. In this paper, a graph class $\mathcal{G}$ is `sparse' if the graphs in $\mathcal{G}$ have bounded maximum density. For a graph class to have bounded expansion, it must be sparse.

Our first result is that sparse string graphs in the plane have linear expansion.

\begin{thm} \label{thm:mainplane} For any integers $d, r \geqslant 0$ and for every string graph $G$ in the plane with maximum density at most $d$,
$$\nabla_r(G) \leqslant 3e((2r + 2) d + 1).$$
\end{thm}

Improving on results of \citet{FP10,FP14}, \citet{Lee17} proved that there exists an absolute constant $c$ such that for any $t \geqslant 1$, every $K_{t,t}$-free string graph in the plane has maximum density at most $ct(\log t)$. Thus \cref{thm:mainplane} implies that $K_{t, t}$-free string graphs in the plane have linear expansion.

\begin{cor} 
There exists an absolute constant $c$ such that for any integers $t \geqslant 1$ and $r \geqslant 0$, for every $K_{t, t}$-free string graph $G$ in the plane, $$\nabla_r(G) \leqslant c(r+1)t(\log t).$$
\end{cor}

\cref{thm:mainplane} is simply the $g=0$ case of the following generalisation for any surface\footnote{The \defn{Euler genus} of a surface with $h$ handles and $c$ cross-caps is $2h + c$. The \defn{Euler genus} of a graph $H$ is the minimum Euler genus of a surface in which $H$ embeds without crossings.}.

\begin{thm} \label{thm:mainsurface} 
For any integers $d, g, r \geqslant 0$, for any surface $\Sigma$ with Euler genus $g$, and for every string graph $G$ in $\Sigma$ with maximum density at most $d$,
$$\nabla_r(G)\leq (\sqrt{3g/2} + 3)e((2r + 2) d + 1).$$
\end{thm}

We in fact prove the following significant generalisation of \cref{thm:mainplane,thm:mainsurface} in the setting of region intersection graphs over proper minor-closed classes.

\begin{thm} \label{main:rigs} 
For any integers $d, r \geqslant 0$ and real number $t \geqslant 1$, for any minor-closed class $\mathcal{G}$ such that every graph in $\mathcal{G}$ has edge density at most $t$, for every region intersection graph $G$ over $\mathcal{G}$ where $G$ has maximum density at most $d$, $$\nabla_{r}(G) \leqslant et((2r + 2) d + 1).$$
\end{thm}

\cref{main:rigs} implies \cref{thm:mainsurface} (and thus \cref{thm:mainplane}) since string graphs in $\Sigma$ are precisely the region intersection graphs over graphs embeddable in $\Sigma$, and it is well-known\footnote{For $n \geqslant 1$, consider an $n$-vertex graph $m$-edge graph with edge density $d$ and Euler genus~$g$. So $d = \frac{m}{n}$. By Euler's formula,  $m\leq 3(g + n - 2)<3(g + n)$. So $d = \frac{m}{n} < \frac{3g}{n} + 3$. If $\frac{3g}{n}\leq \sqrt{3g/2}$ then $d < \sqrt{3g/2} + 3$. Otherwise, $\frac{3g}{n}>\sqrt{3g/2}$, implying $n<\sqrt{6g}$. So $d=\frac{m}{n}\leq \binom{n}{2}/n<\frac{n}{2}<\frac{\sqrt{6g}}{2}=\sqrt{3g/2}$. So $d < \sqrt{3g/2} + 3$ regardless of whether $\frac{3g}{n}\leq \sqrt{3g/2}$ or $\frac{3g}{n}>\sqrt{3g/2}$.} that every graph with Euler genus $g$ has maximum density at most $\sqrt{3g/2} + 3$.

\citet{Kostochka82,Kostochka84} and \citet{Thomason84,Thomason01} proved that there exists an absolute constant $c$ such that every $K_{t}$-minor-free graph has edge density at most $ct\sqrt{\log t}$. Thus \cref{main:rigs} implies the following.

\begin{cor} \label{thm:minorrigs} There exists an absolute constant $c$ such that for any integers $d, r \geqslant 0$ $t \geqslant 1$, for every region intersection graph $G$ over a $K_{t}$-minor-free graph where $G$ has maximum density at most $d$, $$\nabla_r(G) \leqslant ct\sqrt{\log t}\,d(r + 1).$$
\end{cor}

\subsection{Background} \label{Background}

The purpose of this section is to review prior work related to the expansion of string graphs and region intersection graphs. Most of these results depend on the connection between polynomial expansion and balanced sublinear separators\footnote{A \defn{balanced separator} in a graph $G$ is a set $S \subseteq V(G)$ such that each component of $G - S$ has at most $\frac{1}{2}|V(G)|$ vertices.}. \citet{DN16} (also see \citep{Dvorak21,Dvorak16,Dvorak18}) showed that a hereditary\footnote{A class $\mathcal{G}$ of graphs is \defn{hereditary} if for every graph $G \in \mathcal{G}$, every induced subgraph of $G$ is in $\mathcal{G}$.} graph class $\GG$ has polynomial expansion if and only if every graph in $\GG$ admits strongly sublinear separators. The best known bound on the expansion is due to \citet{Dvorak21}, whose proof refines a method of \citet{ER18}.

\begin{thm}[{\citep{Dvorak21,ER18}}]
\label{DER}

Let $\mathcal{G}$ be a hereditary class of graphs such that every $n$-vertex graph in $\GG$ has a balanced separator of size $\mathcal{O}(n^{1 - \varepsilon})$ for some $\varepsilon > 0$. Then for every $G \in \mathcal{G}$, $$\nabla_r(G) \in \mathcal{O}(r^{1/\varepsilon - 1}\polylog r).$$
\end{thm}

On the other hand, it follows from a result of \citet{PRS94} that polynomial expansion implies the presence of strongly sublinear separators (see, for example, \citep[Corollary~2]{Dvorak21}).

There has been a line of research establishing optimal bounds on the size of balanced separators in string graphs. First, \citet{FP10,FP14} proved that every string graph in the plane with $m$ edges has a balanced separator of size $\mathcal{O}(m^{3/4}\sqrt{\log m})$.  \citet{Mat14} improved this bound to $\mathcal{O}(\sqrt{m}\log m)$. \citet{Lee17} established the optimal bound of $\mathcal{O}(\sqrt{m})$. Extending this result, he proved that every region intersection graph with $m$ edges over a $K_{t}$-minor-free graph has a balanced separator of size $\mathcal{O}_{t}(\sqrt{m})$. Since every $n$-vertex graph with maximum density at most $d$ has at most $dn$ edges, this implies that every $n$-vertex region intersection graph $G$ with maximum density at most $d$ over a $K_{t}$-minor-free graph has a balanced separator of size~$\mathcal{O}_{t, d}(\sqrt{n})$. So \cref{DER} implies that such graphs have expansion $\mathcal{O}(r\polylog r)$, which prior to the current work was the best known bound, even for string graphs.

\cref{thm:mainplane,thm:mainsurface,main:rigs,thm:minorrigs} improve this bound to linear, removing the polylogarithmic term, and giving explicit constants with a self-contained purely combinatorial proof, avoiding dependance on the results of \citet{Lee17}, \citet{ER18} and \citet{Dvorak21} and the heavy machinery used in their proofs. In particular, \citet{Lee17} used methods of metric geometry, linear programming duality and multi-flows. \citet{ER18} and \citet{Dvorak21} used expanders and strong results of \citet{CC15} and \citet{SS15}. We use none of these methods or results.

Recently, \citet{Karol25} proved that string graphs with bounded maximum degree in a fixed surface have bounded row treewidth and layered treewidth, which implies that this class has linear expansion by a result of \citet{DMW17}. However, the bound of \citet[Theorem~22]{Karol25} is exponential in the maximum degree. We give much improved bounds on the expansion and work in the setting of bounded maximum density, which is significantly broader than bounded maximum degree. Moreover, sparse string graphs in the plane (with unbounded maximum degree) have unbounded row treewidth and layered treewidth (see, for example, \citep{MSSU24} or \citep[Section~7]{Karol25}). Therefore, row treewidth or layered treewidth cannot be used to prove \cref{thm:mainplane,thm:mainsurface,main:rigs}. 

\cref{main:rigs} is proved in \cref{section:RIGs}. \cref{section:stringgraphs} presents an alternative proof that sparse string graphs in a fixed surface have linear expansion. Specifically, we show that sparse string graphs admit so-called gap-cover-planar drawings, which is of independent interest. \cref{section:optimality} shows that the bound in \cref{thm:mainplane} is within an absolute constant factor of optimal. This implies that the dependence on $r$ and $d$ in \cref{thm:mainsurface} or \cref{main:rigs} cannot be improved to be sublinear. \cref{section:colouring} presents applications of \cref{thm:mainplane,thm:mainsurface,main:rigs} to graph colouring.

\section{Region Intersection Graphs} \label{section:RIGs}

This section proves \cref{main:rigs}. We split our proof into two lemmas.

\subsection{Probabilistic Sampling Lemma}

We now prepare for our first lemma. An \defn{orientation} of an undirected graph $G$ is a directed graph, denoted $\vec{G}$, obtained from $G$ by replacing each edge by one of the two possible arcs with the same endpoints. An edge of $\vec{G}$ directed from $u$ to~$v$ is denoted by $(u, v)$. For a vertex $v \in V(\vec{G})$, let \defn{$N_{\vec{G}}^{-}(v)$} be the set of vertices $u \in V(\vec{G})$ such that $(u, v) \in E(\vec{G})$.

We will use the following classical result of \citet{Hakimi65}. 

\begin{thm} [{\citep{Hakimi65}}]  \label{thm:Hakimi} For any integer $d \geqslant 0$, a graph $G$ has maximum density at most $d$ if and only if there exists an orientation $\vec{G}$ of $G$ such that $|N_{\vec{G}}^{-}(v)| \leqslant d$ for each $v \in V(G)$.
\end{thm}

Let $G$ be a graph, $J$ be a minor of $G$ and $(S_{i} : i \in V(J))$ be a model of $J$ in $G$. For each edge $ij \in E(J)$, let $Q_{ij}$ be a path in $G[S_{i} \cup S_{j}]$ with one endpoint in $S_{i}$ and another endpoint in $S_{j}$. We say that such a collection of paths $(Q_{ij} : ij \in E(J))$ \defn{represents} the model $(S_{i} : i \in V(J))$ if for each $i \in V(J)$, the graph $G[S_{i} \cap \bigcup_{ij \in E(J)}Q_{ij}]$ is connected. Let $\vec{G}$ be an orientation of $G$, and let $J'$ be a subgraph of $J$. A \defn{junction} in $J'$ is a pair of edges $((a, b), ij)$ such that $(a, b) \in E(\vec{G})$, $ij \in E(J')$, $b \in Q_{ij}$ and $a \in S_{\ell}$ for some $\ell \in V(J') \setminus \{i, j\}$. See \cref{figure:junction}.

\begin{figure}[H]
        \centering
        \scalebox{1.3}{\includegraphics{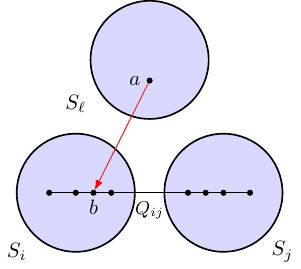}}
        \caption{A junction in $J'$: a pair of edges $(a, b) \in E(\vec{G})$ and $ij \in E(J')$ such that $b \in Q_{ij}$ and $a \in S_{\ell}$. The vertices $i, j, \ell \in V(J')$ are distinct.}
        \label{figure:junction}
\end{figure}

We say that $J'$ is \defn{junction-free} if there exists no junction in $J'$. Note that the definition of `junction' depends on $\vec{G}$, $J$, the model $(S_{i} : i \in V(J))$, the collection of paths $(Q_{ij} : ij \in E(J))$ that represents this model, and the subgraph $J'$ of $J$.

We are now ready to state our first lemma. The proof uses a probabilistic method similar to one described by \citet{Sharir03}. Rom Pinchasi (see \citep[Lemma~4.1]{AFPS14}) used this approach to bound the edge density of $k$-cover-planar graphs. Using the same approach, \citet{Wood25} bounded the edge density of $k$-gap-cover-planar graphs, which  is an essential ingredient in his proof that $k$-gap-cover-planar graphs have linear expansion. We use this result in \cref{section:stringgraphs} as part of our alternative proof that sparse string graphs have linear expansion. See \cref{section:stringgraphs} for definitions of $k$-cover-planar and $k$-gap-cover-planar graphs.

\begin{lem} \label{lem:probabilistic} Let $J$ be a minor of a graph $G$ and $(S_{i} : i \in V(J))$ be a model of $J$ in $G$. Let $(Q_{ij} : ij \in E(J))$ be a collection of paths that represents $(S_{i} : i \in V(J))$ such that each path $Q_{ij}$ has at most $k$ vertices. Let $\vec{G}$ be an orientation of $G$ such that $|N_{\vec{G}}^{-}(v)| \leqslant d$ for each $v \in V(\vec{G})$. Suppose that every junction-free subgraph of $J$ has edge density at most some real number $\beta$. Then $J$ has edge density at most $\beta e(kd + 1)$.
\end{lem}

\begin{proof} We may assume that $V(J) \neq \emptyset$, otherwise \cref{lem:probabilistic} is trivial. For each $ij \in E(J)$, define $R_{ij}$ to be the set of vertices $\ell \in V(J) \setminus \{i, j\}$ such that there exists a junction $((a, b), ij)$ in $J$ with $a \in S_{\ell}$ (so $b \in Q_{ij}$ and $(a, b) \in E(\vec{G})$). Observe that $|R_{ij}| \leqslant kd$ since $(S_i: i \in V(J))$ are pairwise disjoint, $|Q_{ij}| \leqslant k$ and $|N_{\vec{G}}^{-}(b)| \leqslant d$ for each $b \in Q_{ij}$. 

Define $\eta_{a} :=  \frac{(a + 1)^{a + 1}}{a^a}$ for each integer $a \geqslant 0$. Note that $\eta_{a} < e(a + 1)$.

Choose each vertex of $J$ independently with
probability $p := \frac{1}{kd + 1}$. Let $J'$ be the subgraph of $J$ where $V(J')$ is the set of chosen vertices, and $E(J')$ is the set of edges $ij \in E(J)$ such that $i$ and $j$ are chosen, but no vertex of $R_{ij}$ is chosen. By definition of $R_{ij}$, every possible graph $J'$ is junction-free. By assumption, every possible non-empty graph $J'$ satisfies $|E(J')|/|V(J')| \leqslant \beta$. Let $n^*$ and $m^*$ be the expected value of $|V(J')|$ and $|E(J')|$ respectively. Since the number of possible graphs $J'$ is finite, $\frac{m^*}{n^*} \leqslant \beta$. By definition, $n^* = p|V(J)|$. The probability that an edge $ij \in E(J)$ is in $J'$ equals $p^{2}(1 - p)^{|R_{ij}|} \geqslant p^{2}(1 - p)^{kd}$. Therefore $m^* \geqslant p^{2}(1 - p)^{kd}|E(J)|$. Hence $p^{2}(1 - p)^{kd}|E(J)| \leqslant m^* = \frac{m^*}{n^{*}}n^* = \frac{m^*}{n^{*}}p|V(J)|$. By the choice of $p$, 
\begin{equation*}
\frac{|E(J)|}{|V(J)|} \leqslant 
\frac{m^*}{n^{*}}\frac{1}{p(1 - p)^{kd}} =
\frac{m^*}{n^{*}}\eta_{kd} <
\frac{m^*}{n^{*}}e(kd + 1) \leqslant
\beta e(kd + 1). \qedhere
\end{equation*}
\end{proof}

\cref{lem:probabilistic} or its variants are useful for establishing tight bounds on the expansion of various graph classes. For example, let $\mathcal{G}$ be a sparse class of graphs, so there exists a non-negative integer $d$ such that the graphs in $\mathcal{G}$ have maximum density at most $d$. By \cref{thm:Hakimi}, for any graph $G \in \mathcal{G}$, there exists an orientation $\vec{G}$ of $G$ such that $|N_{\vec{G}}^{-}(v)| \leqslant d$ for each $v \in V(\vec{G})$.  Let $J$ be an $r$-shallow minor of a graph $G$ and $(S_{i} : i \in V(J))$ be an $r$-shallow model of $J$ in $G$. Let $(Q_{ij} : ij \in E(J))$ be a collection of paths that represents $(S_{i} : i \in V(J))$ such that each path $Q_{ij}$ has $\mathcal{O}(r)$ vertices. If every junction-free subgraph of $J$ has edge density bounded by a function independent of $r$, then by \cref{lem:probabilistic}, $J$ has edge density $\mathcal{O}_d(r)$, implying that $\mathcal{G}$ has linear expansion. In particular, this is how we use \cref{lem:probabilistic} to prove \cref{main:rigs}.

Note that some assumptions in \cref{lem:probabilistic} can be relaxed. In particular, the assumptions `$G[S_{i} \cap \bigcup_{ij \in E(J)}Q_{ij}]$ is connected' and `each path $Q_{ij}$ has one endpoint in $S_{i}$ and another in $S_{j}$' are not needed. Moreover, $(Q_{ij} : ij \in E(J))$ can be vertex-sets, not necessarily paths. Also, the assumption `$|N_{\vec{G}}^{-}(v)| \leqslant d$ for each $v \in V(\vec{G})$' can be relaxed to `for each $ij \in E(J)$ and for each $b \in Q_{ij}$, there are at most $d$ vertices $a \in V(G) \setminus (S_{i} \cup S_{j})$ such that $(a, b) \in E(\vec{G})$'. We present \cref{lem:probabilistic} in the described setting since it is more intuitive this way, and this matches its application in the proof of \cref{main:rigs}.

\subsection{Model Construction Lemma} \label{section:secondlemma}

We now prepare for our second lemma.

Let $G$ be a region intersection graph over a graph $H$. So there exists a collection $(A_{v} \subseteq H: v \in V(G))$ of non-empty trees in $H$ such that $uv \in E(G)$ if and only if $V(A_u) \cap V(A_v) \neq \emptyset$. Note that every tree $A_{v}$ is a subgraph of $H$, but not necessarily induced. Let $J$ be a minor of $G$ and $(S_{i} : i \in V(J))$ be a model of $J$ in $G$. For each $i \in V(J)$, fix a vertex $c_{i} \in S_{i}$.

We choose a collection of paths that represents the model $(S_{i} : i \in V(J))$ in the following specific way. For each $i \in V(J)$, let $T_{i}$ be a tree of $G[S_{i}]$ rooted at $c_{i}$ such that $V(T_{i}) = S_{i}$. For each $a \in S_{i} \setminus \{c_{i}\}$, let $a^{\uparrow}$ be the parent of $a$ in $T_{i}$. Let $\preccurlyeq$ be an arbitrary vertex ordering of~$J$. Let $ij \in E(J)$ be an edge with $i \preccurlyeq j$. To define a path $Q_{ij}$, we choose an edge $xy \in E(G)$ with $x \in S_{i}$ and $y \in S_{j}$. Since $(S_{i} : i \in V(J))$ is a model, such an edge $xy$ exists. To choose $xy$, we consider three cases.

\textbf{Case 1.} $c_{i}$ is adjacent to no vertex of $S_{j}$:

Let $\alpha$ be the minimal integer such that there exists an edge $x'y' \in E(G)$ with $x' \in S_{i}$, $y' \in S_{j}$, and $\dist_{G[S_{i}]}(c_{i}, x') = \alpha$. Choose $x$ and $y$ such that $\dist_{G[S_{i}]}(c_{i}, x) = \alpha$ and $\dist_{A_{x}}(A_{x} \cap A_{x^{\uparrow}}, A_{x} \cap A_{y})$ is minimised. By assumption, $x \neq c_{i}$, and hence $x^{\uparrow}$ exists; so this choice is well-defined.

\textbf{Case 2.} $c_{i}$ is adjacent to a vertex of $S_{j}$, but $c_{i}c_{j} \notin E(G)$:

Let $x := c_{i}$. Choose $y$ such that (a) $\dist_{G[S_{j}]}(c_{j}, y)$ is minimised, and (b) (subject to (a)) $\dist_{A_{y}}(A_{y} \cap A_{y^{\uparrow}}, A_{y} \cap A_{c_{i}})$ is minimised. By assumption, $y \neq c_{j}$, and hence $y^{\uparrow}$ exists; so this choice is well-defined.

\textbf{Case 3.} $c_{i}c_{j} \in E(G)$:

In this case, let $x := c_{i}$ and $y := c_{j}$.

Thus the edge $xy \in E(G)$ is chosen. Let $P_{i, j}$ be the path in $T_{i}$ between $c_{i}$ and $x$, and let $P_{j, i}$ be the path in $T_{j}$ between $y$ and $c_{j}$. Define $Q_{ij}$ to be the path that consists of the concatenation of  $P_{i, j}$, the edge $xy$, and $P_{j, i}$ (see \cref{figure:path}).

\begin{figure}[h]
        \centering
        \scalebox{1.3}{\includegraphics{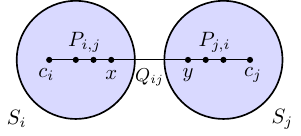}}
        \caption{$Q_{ij}$ is the path that consists of the concatenation of  $P_{i, j}$, the edge $xy$, and $P_{j, i}$.}
        \label{figure:path}
\end{figure}

The collection of paths $(Q_{ij} : ij \in E(J))$ defined above \defn{properly represents} the model $(S_{i} : i \in V(J))$. For each $ij \in E(J)$, we have $S_{i} \cap Q_{ij} = P_{i, j}$ and $c_{i} \in P_{i, j}$. Therefore, for each $i \in V(J)$, the graph $G[S_{i} \cap \bigcup_{ij \in E(J)}Q_{ij}]$ is connected. Hence $(Q_{ij} : ij \in E(J))$ represents $(S_{i} : i \in V(J))$. So the definition of `junction' can be applied for $J$ and an orientation $\vec{G}$ of $G$. Note that the definition of `properly represents' depends on $G$, $H$, the collection $(A_{v} \subseteq H: v \in V(G))$, $J$, the model $(S_{i} : i \in V(J))$ of $J$ in $G$, the vertices $c_{i} \in S_{i}$, the trees $T_{i}$ rooted at $c_{i}$, and the ordering $\preccurlyeq$. 

Observe that for any subgraph $J'$ of $J$, $(S_{i} : i \in V(J'))$ is a model of $J'$ in $G$, and $(Q_{ij} : ij \in E(J'))$ properly represents $(S_{i} : i \in V(J'))$.

\begin{lem} \label{lem:model} Let $G$ be a region intersection graph over a graph $H$, $J$ be a minor of $G$, and $(S_{i} : i \in V(J))$ be a model of $J$ in $G$. Let $(Q_{ij} : ij \in E(J))$ be a collection of paths that properly represents the model $(S_{i} : i \in V(J))$. Assume that $G$ has an orientation $\vec{G}$ such that $J$ is junction-free. Let $J^*$ be the subgraph of $J$ induced by $\{v \in V(J): \deg_{J}(v) \geqslant 2\}$. Then $J^*$ is a minor\footnote{Note that $J$ might not be a minor of $H$. For example, $G = K_{2}$ is a region intersection graph over $H = K_{1}$, where $J = G$ is junction-free for the trivial orientation of $G$, but $J$ is not a minor of $H$.} of $H$.
\end{lem}

\begin{proof} We employ all the notation from the start of \cref{section:secondlemma} that contributes to the definition of `properly represents'. In particular, we use notation $(A_{v} \subseteq H: v \in V(G))$, $c_{i}$, $T_{i}$, $\preccurlyeq$, $\alpha$, $x$, $y$, $P_{i, j}$, $P_{j, i}$.

We further assume that $V(J^*) \neq \emptyset$, otherwise \cref{lem:model} is trivial. For the sake of convenience, assume that $V(J^*) = \{1, 2, \dots, |V(J^*)|\}$ where $i \leqslant j$ if and only if $i \preccurlyeq j$ for any $i, j \in V(J^*)$.

\paragraph{Proof sketch:} The proof constructs branch sets $B_{1}, \dots, B_{|V(J^*)|} \in V(H)$ of a model of $J^*$ in $H$. To do so, for each edge $ij \in E(J^*)$ with $i < j$ we construct an auxiliary set $B_{i, j} \subseteq \bigcup_{t \in P_{i, j}}V(A_{t})$. Informally, $B_{i, j}$ will be the part of the branch set $B_{i}$ that is `responsible' from the viewpoint of $B_{i}$ for touching $B_{j}$ (although it is not necessary that $B_{i, j}$ touches~$B_{j}$). Having such sets $B_{i, j}$ for each edge $ij \in E(J^*)$ with $i < j$ constructed, we then define $B_{i}^{<}:= V(A_{c_{i}}) \cup  \bigcup_{ij \in E(J^*): i < j}B_{i, j}$ for each $i \in \{1, \dots, |V(J^*)|\}$. We show that $B_{1}^{<}, \dots, B_{|V(J^*)|}^{<}$ are connected in $H$ and pairwise disjoint. Using the construction of the sets $B_{i}^{<}$, we inductively construct a set $B_{j, i} \subseteq \bigcup_{t \in P_{j, i}}V(A_{t})$ for each edge $ij \in \{1, \dots, |V(J^*)|\}$ with $i < j$. Informally, $B_{j, i}$ will be the part of the branch set $B_{j}$ that is `responsible' from the viewpoint of $B_{j}$ for touching~$B_{i}$. We define $B_{j}^{>} := \bigcup_{ij \in E(J^*): i < j}B_{j, i}$ and $B_{j}:= B_{j}^{<} \cup B_{j}^{>}$. Then we inductively show that $B_{1}, \dots, B_{|V(J^*)|}$ form branch sets of a model of $J^*$ in~$H$.

\begin{claim} \label{claim:basic} $\{c_{1}, \dots, c_{|V(J^*)|}\}$ is an independent set in $G$, and thus $V(A_{c_{1}}), \dots, V(A_{c_{|V(J^*)|}})$ are pairwise disjoint.
\end{claim}

\begin{proof} Suppose for the sake of contradiction that $c_{i}c_{j} \in E(G)$ for some distinct $i, j \in \{1, 2, \dots, |V(J^*)|\}$. Since $i \in V(J^*)$, we have $\deg_{J}(i) \geqslant 2$, and hence there exists an edge $ik \in E(J)$ such that $k \neq j$. Similarly, there exists an edge $j\ell \in E(J)$ such that $i \neq \ell$. Recall that $c_{i} \in Q_{ik}$ and $c_{j} \in Q_{j\ell}$. If $(c_{i}, c_{j}) \in E(\vec{G})$ then $((c_{i}, c_{j}), j\ell)$ is a junction in $J$. Otherwise, $(c_{j}, c_{i}) \in E(\vec{G})$, and so $((c_{j}, c_{i}), ik)$ is a junction in $J$. This contradicts our starting assumption that $J$ is junction-free.
\end{proof}

\paragraph{Construction of $\bm{B_{i,j}}$ for $\bm{i < j}$:}

 Let $ij \in E(J^*)$ be an edge of $J^*$ such that $i < j$. To define $B_{i, j}$, we consider cases that correspond to the cases in the definition of $Q_{ij}$. By \cref{claim:basic}, Case 3 (where $c_{i}c_{j} \in E(G)$) does not occur. Recall the notation $x, y, P_{i, j}, P_{j, i}$ from the definition of $Q_{ij}$ (see \cref{figure:path}).

\textbf{Case 1.} $c_{i}$ is adjacent to no vertex of $S_{j}$:

In this case $x \neq c_{i}$, and hence $x^{\uparrow}$ exists. As illustrated in \cref{figure:Bij}, let $B_{i, j}$ be the set of vertices of $H$ that consists of $\bigcup_{t \in P_{i, j} \setminus \{x\}}V(A_{t})$ and the vertices of $A_{x}$ that constitute the shortest path in $A_{x}$ between $V(A_{x}) \cap V(A_{x^{\uparrow}})$ and $V(A_{x}) \cap V(A_{y})$ except the endpoint in $V(A_{x}) \cap V(A_{y})$. Note that there exists an edge of $H$ between $B_{i, j}$ and $V(A_{y})$, drawn red in \cref{figure:Bij}.

\begin{figure}[h]
        \centering
        \scalebox{1.3}{\includegraphics{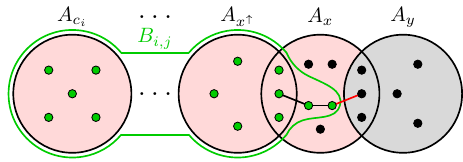}}
        \caption{Construction of $B_{i, j}$ in Case 1. The vertices of $B_{i, j}$ are green. An edge of $H$ between $B_{i, j}$ and $V(A_{y})$ is red.}
        \label{figure:Bij}
\end{figure}

Recall the choice of $x$ and $y$ in the corresponding Case 1 in the definition of $Q_{ij}$ and the notation of $\alpha$. Since $\dist_{G[S_{i}]}(c_{i}, x) = \alpha$, we have $(\bigcup_{t \in P_{i, j} \setminus \{x\}}V(A_{t})) \cap \bigcup_{r \in S_{j}} V(A_{r}) = \emptyset$. Since $\dist_{A_{x}}(A_{x} \cap A_{x^{\uparrow}}, A_{x} \cap A_{y})$ is minimised, no vertex of the shortest path in $A_{x}$ between $V(A_{x}) \cap V(A_{x^{\uparrow}})$ and $V(A_{x}) \cap V(A_{y})$ except the endpoint in $V(A_{x}) \cap V(A_{y})$ belongs to $\bigcup_{r \in S_{j}} V(A_{r})$. Thus in this case, \begin{equation*} \label{specialproperty} \tag{$\star$} B_{i, j} \cap \bigcup_{r \in S_{j}} V(A_{r}) = \emptyset.
\end{equation*}

\textbf{Case 2.} $c_{i}$ is adjacent to a vertex of $S_{j}$, but $c_{i}c_{j} \notin E(G)$:

Recall that in this case $x = c_{i}$, so $P_{i, j}$ consists of the single vertex $x$. Define $B_{i, j} := V(A_{c_{i}})$.

Observe that in each case, $V(A_{c_{i}}) \subseteq B_{i, j}$, and $B_{i, j} \subseteq \bigcup_{t \in P_{i, j}}V(A_{t})$, and $H[B_{i, j}]$ is connected.

\paragraph{Definition and properties of $\bm{B_{i}^{<}}$:} For each $i \in V(J^*)$, let $$B_{i}^{<}:= V(A_{c_{i}}) \cup \bigcup_{ij \in E(J^*): i < j}B_{i, j}.$$ Note that if there is no edge $ij \in E(J^*)$ such that $i < j$, then $B_{i}^{<} = V(A_{c_{i}})$. Recall that $V(A_{c_{i}}) \subseteq B_{i, j}$ and $H[B_{i, j}]$ is connected for each edge $ij \in E(J^*)$ such that $i < j$. Hence $H[B_{i}^{<}]$ is connected for each $i \in V(J^*)$.

\begin{claim} \label{claim:<disjointness} $B_{1}^{<}, ..., B_{|V(J^*)|}^{<}$ are pairwise disjoint.
\end{claim}

\begin{proof}

Assume for the sake of contradiction that $B_{i}^{<} \cap B_{j}^{<} \neq \emptyset$ for some $i, j \in V(J^*)$ with $i < j$.

\textbf{Case 1.} $V(A_{c_{i}}) \cap B_{j}^{<} \neq \emptyset$:

By \cref{claim:basic}, $V(A_{c_{i}}) \cap V(A_{c_{j}}) = \emptyset$. So there exists an edge $j\ell \in E(J^*)$ such that $j < \ell$ and $V(A_{c_{i}}) \cap (B_{j, \ell} \setminus V(A_{c_{j}})) \neq \emptyset$. Since $B_{j, \ell} \subseteq \bigcup_{t \in P_{j, \ell}}V(A_{t})$,  there exists $t \in P_{j, \ell}$ such that $V(A_{t}) \cap V(A_{c_{i}}) \neq \emptyset$, implying $tc_{i} \in E(G)$. Since $i \in V(J^*)$, we have $\deg_{J}(i) \geqslant 2$, and hence there exists an edge $ik \in E(J)$ such that $k \neq j$, and so $c_{i} \in Q_{ik}$. If $(t, c_{i}) \in E(\vec{G})$, then $((t, c_{i}), ik)$ is a junction in $J$. If $(c_{i}, t) \in E(\vec{G})$ then $((c_{i}, t), j\ell)$ is a junction in $J$ because $i, j, \ell$ are distinct (since $i < j < \ell$).

\textbf{Case 2.} $V(A_{c_{i}}) \cap B_{j}^{<} = \emptyset$:

Let $h \in B_{i}^{<} \cap B_{j}^{<}$. By assumption and the definition of $B_{i}^{<}$, there exists an edge $ik \in E(J^*)$ such that $i < k$ and $h \in (B_{i, k} \setminus V(A_{c_{i}})) \cap B_{j}^{<}$. Since $B_{i, k} \subseteq \bigcup_{t \in P_{i, k}}V(A_{t})$, there exists $t \in P_{i, k}$ such that $h \in V(A_{t}) \cap B_{j}^{<}$.

First, suppose that $h \in V(A_{c_{j}})$. So $h \in V(A_{t}) \cap V(A_{c_{j}})$, and hence $tc_{j} \in E(G)$. Since $j \in V(J^*)$, we have $\deg_{J}(j) \geqslant 2$, and hence there exists an edge $j\ell \in E(J)$ such that $\ell \neq i$. If $(t, c_{j}) \in E(\vec{G})$ then $((t, c_{j}), j\ell)$ is a junction in $J$. Therefore $(c_{j}, t) \in E(\vec{G})$. Since $((c_{j}, t), ik)$ is not a junction in $J$, we have $k = j$. Recall that $h \in V(A_{c_{j}})$ and $h \in B_{i, k} = B_{i, j}$, and hence $B_{i, j} \cap V(A_{c_{j}}) \neq \emptyset$. If Case 1 in the construction of $B_{i, j}$ holds, then this violates \cref{specialproperty}. Otherwise, Case 2 in the construction of $B_{i, j}$ holds, then $B_{i, j} = V(A_{c_{i}})$, implying $V(A_{c_{i}}) \cap V(A_{c_{j}}) \neq \emptyset$, contradicting \cref{claim:basic}.

So $h \notin V(A_{c_{j}})$. Since $h \in B_{j}^{<}$, there exists an edge $j\ell \in E(J^*)$ such that $j < \ell$ and $h \in B_{j, \ell} \setminus V(A_{c_{j}})$. Since $B_{j, \ell} \subseteq \bigcup_{r \in P_{j, \ell}}V(A_{r})$, there exists $r \in P_{j, \ell}$ such that $h \in V(A_{t}) \cap V(A_{r})$, implying $tr \in E(G)$. If $(t, r) \in E(\vec{G})$ then $((t, r), j\ell)$ is a junction in $J$ because $i, j, \ell$ are distinct (since $i < j < \ell$). Therefore $(r, t) \in E(\vec{G})$. Since $((r, t), ik)$ is not a junction in $J$, we have $k = j$. Recall that $h \in V(A_{r})$ and $h \in B_{i, k} = B_{i, j}$, and $h \notin V(A_{c_i})$. So $(B_{i, j} \cap V(A_{r})) \setminus V(A_{c_{i}}) \neq \emptyset$. Recall that $r \in P_{j, \ell} \subseteq S_{j}$. If Case 1 in the construction of $B_{i, j}$ holds, then this violates \cref{specialproperty}. Otherwise, Case 2 in the construction of $B_{i, j}$ holds, then $B_{i, j} = V(A_{c_{i}})$, a contradiction.
\end{proof}

\paragraph{\boldmath Construction of $B_{1}, \dots, B_{|V(J^*)|}$, base case:}

We now expand $B_{1}^{<}, ..., B_{|V(J^*)|}^{<}$ into `real' branch sets $B_{1}, \dots, B_{|V(J^*)|}$. We do this inductively. By induction on $j \in \{1, \dots, |V(J^*)|\}$, we construct a set $B_{j} \subseteq V(H)$ such that: 
\begin{enumerate}[label=(\roman*)]
    \item \label{induction1} $B_{1}, \dots, B_{j}, B_{j + 1}^{<}, \dots, B_{|V(J^*)|}^{<}$ are pairwise disjoint, non-empty and connected in $H$,

    \item \label{induction2} $B_{1}, \dots, B_{j}$ constitute branch sets of a model of $J^*[\{1, \dots, j\}]$ in $H$, and

    \item \label{induction3} $B_{i}^{<} \subseteq B_{i}$ for any $i \in \{1, \dots, j\}$.
\end{enumerate}

Note that for $j = |V(J^*)|$, item \ref{induction2} shows that $J^*$ is a minor of $H$, as required.

In the base case, define $B_{1} := B_{1}^{<}$. Recall that $H[B_{i}^{<}]$ is non-empty and connected for each~$i \in V(J^*)$. By \cref{claim:<disjointness}, the induction statement holds for this choice.

\paragraph{Construction of $\bm{B_{j,i}}$ for $\bm{i < j}$:}

Fix $j \geqslant 2$ (through to the end of the proof). Assume that the sets $B_{1}, \dots, B_{j - 1}$ are constructed such that \ref{induction1}, \ref{induction2} and \ref{induction3} hold. Our goal is to construct~$B_{j}$.

For each $ij \in E(J^*)$ such that $i < j$, we construct an auxiliary set $B_{j, i} \subseteq V(H)$. To do so, we consider two cases.

\textbf{Case A.} $B_{i} \cap \bigcup_{t \in P_{j, i}}V(A_{t}) = \emptyset$:

Define $B_{j, i} := \bigcup_{t \in P_{j, i}}V(A_{t})$.

\textbf{Case B.} $B_{i} \cap \bigcup_{t \in P_{j, i}}V(A_{t}) \neq \emptyset$:

Let $z$ be the closest vertex to $c_{j}$ in the path $P_{j, i}$ such that $B_{i} \cap V(A_{z}) \neq \emptyset$ (by assumption, such vertex $z$ exists). By item \ref{induction1} of the induction hypothesis, $B_{i} \cap B_{j}^{<} = \emptyset$. Recall that $V(A_{c_{j}}) \subseteq B_{j}^{<}$. Therefore, $B_{i} \cap V(A_{c_{j}}) = \emptyset$. Consequently, $z \neq c_{j}$ and hence the parent $z^{\uparrow}$ of $z$ in $P_{j, i}$ exists. Let $L_{j, i}$ be the subpath of $P_{j, i}$ between $c_{j}$ and $z^{\uparrow}$. As illustrated in \cref{figure:Bji}, let $B_{j, i}$ be the set of vertices of $H$ that consists of $\bigcup_{t \in L_{j, i}}V(A_{t})$ and the vertices of $A_{z}$ that constitute the shortest path in $A_{z}$ between $V(A_{z}) \cap V(A_{z^{\uparrow}})$ and $V(A_{z}) \cap B_{i}$ except the endpoint in $V(A_{z}) \cap V(B_{i})$.

\begin{figure}[h]
        \centering
        \scalebox{1.3}{\includegraphics{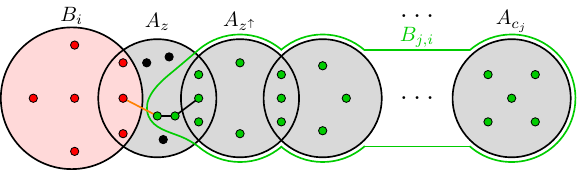}}
        \caption{Construction of $B_{j, i}$ in Case B. The vertices of $B_{j, i}$ are green. The vertices of $B_{i}$ are red. An edge of $H$ between $B_{i}$ and $B_{j, i}$ is orange.}
        \label{figure:Bji}
\end{figure}

Observe that in each case, $V(A_{c_{j}}) \subseteq B_{j, i}$, and $B_{j, i} \subseteq \bigcup_{t \in P_{j, i}}V(A_{t})$, and $H[B_{j, i}]$ is connected.

\begin{claim} \label{claim:Bji} For each edge $ij \in E(J^*)$ such that $i < j$, we have $B_{j, i} \cap B_{i} = \emptyset$, and there is an edge of $H$ between $B_{j, i}$ and $B_{i}$.
\end{claim}

\begin{proof} In Case B, both properties immediately follow from the construction of $B_{j, i}$. Now assume that Case A holds. By this assumption and the choice of $B_{j, i}$, we have $B_{j, i} \cap B_{i} = \emptyset$.

By definition of $B_{i}^{<}$, we have $B_{i, j} \subseteq B_{i}^{<}$. By item \ref{induction3} of the induction hypothesis, $B_{i}^{<} \subseteq B_{i}$. Therefore $B_{i, j} \subseteq B_{i}$.

Recall that $x$ and $y$ are the endpoints of the paths $P_{i, j}$ and $P_{j, i}$ respectively and $xy \in E(G)$. Recall that the construction of $B_{i, j}$ is based on two cases. In Case 1, there is an edge of $H$ between $B_{i, j}$ and $V(A_{y})$ by the construction of $B_{i, j}$ (see the red edge in \cref{figure:Bij}). In Case 2, $c_{i}$ is adjacent to a vertex of $S_{j}$, and hence $c_{i}y \in E(G)$ due to the choice of $y$ in the definition of $Q_{ij}$. So in Case 2, $V(A_{c_{i}}) \cap V(A_{y}) \neq \emptyset$. This condradicts the assumption of Case A because $V(A_{c_{i}}) = B_{i, j} \subseteq B_{i}$ and $y \in P_{j, i}$. Thus there is an edge of $H$ between $B_{i, j}$ and $V(A_{y})$. By the definition of $B_{j, i}$ in Case A, $V(A_{y}) \subseteq B_{j, i}$. Since $B_{i, j} \subseteq B_{i}$, there is an edge between $B_{i}$ and~$B_{j, i}$.
\end{proof}

\paragraph{Construction of $\bm{B_{j}}$:}

Let $B_{j}^{>} := \bigcup_{ij \in E(J^*): i < j}B_{j, i}$ and $B_{j}:= B_{j}^{<} \cup B_{j}^{>}$. Item \ref{induction3} of the induction is immediately satisfied.

Recall that $V(A_{c_{j}}) \subseteq  B_{j}^{<}$ and $H[B_{j}^{<}]$ is connected. Recall that for each edge $ij \in E(J^*)$ such that $i < j$, we have $V(A_{c_{j}})  \subseteq B_{j, i}$ and $H[B_{j, i}]$ is connected. Thus $H[B_{j}]$ is non-empty and connected. The following two claims prove that $B_{j}$ is disjoint from $B_{i}$ for any $i \in \{1, \dots, j - 1\}$, and that $B_{j}$ is disjoint from $B_{k}^{<}$ for any $k \in \{j + 1, \dots, |V(J^*)|\}$. This implies that item \ref{induction1} of the induction is also satisfied.

\begin{claim} \label{claim:first} $B_{j} \cap B_{i} = \emptyset$ for any $i \in \{1, \dots, j - 1\}$.
\end{claim}

\begin{proof} We prove the following auxiliary fact: for each edge $i_{0}j \in E(J^*)$ such that $i_{0} < j$, we have  $(B_{j, i_0} \setminus V(A_{c_{j}})) \cap B_{i} = \emptyset$. If $i_{0} = i$, then this follows from \cref{claim:Bji}.

Now assume that $i_{0} \neq i$. Suppose for the sake of contradiction that there exists $h \in (B_{j, i_0} \setminus V(A_{c_{j}})) \cap B_{i}$. Recall that $B_{j, i_{0}} \subseteq \bigcup_{t \in P_{j, i_{0}}}V(A_{t})$. So there exists $t \in P_{j, i_{0}}$ such that $h \in V(A_{t})$.

First, suppose that $h \in V(A_{c_i})$. So $h \in V(A_{c_i}) \cap V(A_{t})$, implying $c_{i}t \in E(G)$. Since $i \in V(J^*)$, we have $\deg_{J}(i) \geqslant 2$, and hence there exists an edge $i\ell \in E(J)$ such that $\ell \neq j$. If $(t, c_{i}) \in E(\vec{G})$ then $((t, c_{i}), i\ell)$ is a junction in $J$. Otherwise, $(c_{i}, t) \in E(\vec{G})$, and so $((c_{i}, t), ji_0)$ is a junction in $J$ because $i_{0} \neq i$.

So $h \notin V(A_{c_i})$. By definition of $B_{i}$, we have $B_{i} = V(A_{c_{i}}) \cup \bigcup_{k: ik \in E(J^*)}B_{i, k}$. Since $h \in B_{i}$ and $h \notin V(A_{c_i})$, we have $h \in B_{i, k}$ for some edge $ik \in E(J^*)$. Recall that $B_{i, k} \subseteq \bigcup_{z \in P_{i, k}}V(A_{z})$ regardless of whether $i < k$ or $k < i$. So there exists a vertex $z \in P_{i, k}$ such that $h \in V(A_{z})$. Since $h \in V(A_{z})$ and $h \in V(A_{t})$, we have $V(A_{z}) \cap V(A_{t}) \neq \emptyset$, implying $zt \in E(G)$. If $(z, t) \in E(\vec{G})$ then $((z, t), ji_{0})$ is a junction in $J$ because $i_{0} \neq i$. Therefore $(t, z) \in E(\vec{G})$. Since $((t, z), ik)$ is not a junction in $J$, we have $k = j$. Hence $h \in (B_{i, j} \setminus V(A_{c_{i}}))  \cap V(A_{t})$, implying $(B_{i, j} \setminus V(A_{c_{i}}))  \cap V(A_{t}) \neq \emptyset$. Recall the construction of $B_{i, j}$ where $i < j$. In Case 1, the property \cref{specialproperty} contradicts $(B_{i, j} \setminus V(A_{c_{i}}))  \cap V(A_{t}) \neq \emptyset$. In Case 2, we have $B_{i, j} = V(A_{c_{i}})$, a contradiction to $(B_{i, j} \setminus V(A_{c_{i}}))  \cap V(A_{t}) \neq \emptyset$. So in each case we have reached a contradiction. 

We have shown that for each edge $i_{0}j \in E(J^*)$ such that $i_{0} < j$, we have  $(B_{j, i_0} \setminus V(A_{c_{j}})) \cap B_{i} = \emptyset$. By the definition of $B_{j}^{>}$, this implies $(B_{j}^{>} \setminus V(A_{c_{j}})) \cap B_{i} = \emptyset$.  By item \ref{induction1} of the induction hypothesis, $B_{j}^{<} \cap B_{i} = \emptyset$. Recall that $V(A_{c_{j}}) \subseteq B_{j}^{<}$ and $B_{j} = B_{j}^{<} \cup B_{j}^{>}$. Thus $B_{j} \cap B_{i} = \emptyset$, as desired.
\end{proof}

\begin{claim} \label{claim:second} $B_{j} \cap B_{k}^{<} = \emptyset$ for any $k \in \{j + 1, \dots, |V(J^*)|\}$.
    
\end{claim}

\begin{proof} We prove the following auxiliary fact:  for each edge $ij \in E(J^*)$ such that $i < j$, we have $B_{j, i} \cap B_{k}^{<} = \emptyset$.

Assume for the sake of contradiction that there exists $h \in B_{j, i} \cap B_{k}^{<}$. Recall that $B_{j, i} \subseteq \bigcup_{t \in P_{j, i}}V(A_{t})$. Therefore $h \in V(A_{t})$ for some $t \in P_{j, i}$.

First, suppose that $h \in V(A_{c_{k}})$. So $h \in V(A_{c_{k}}) \cap V(A_{t})$, implying $c_{k}t \in E(G)$. If $(c_{k}, t) \in E(\vec{G})$ then $((c_{k}, t), ji)$ is a junction since $i < j < k$. Otherwise, $(t, c_{k}) \in E(\vec{G})$. Since $k \in V(J^*)$, we have $\deg_{J}(k) \geqslant 2$, and hence there exists an edge $kq \in E(J)$ such that $q \neq j$. Then $((t, c_{k}), kq)$ is a junction.

So $h \notin V(A_{c_{k}})$. Then by definition of $B_{k}^{<}$, we have $h \in B_{k, \ell}$ for some edge $k\ell \in E(J^*)$ such that $k < \ell$. So $i < j < k < \ell$. Recall that $B_{k, \ell} \subseteq \bigcup_{z \in P_{k, \ell}}V(A_{z})$. So there exists $z \in P_{k, \ell}$ such that $h \in V(A_{z})$. Therefore $h \in V(A_{z}) \cap V(A_{t})$, implying $zt \in E(G)$. If $(z, t) \in E(\vec{G})$ then $((z, t), ji)$ is a junction in $J$ because $i, j, k$ are distinct. If $(t, z) \in E(\vec{G})$ then $((t, z), k\ell)$ is a junction in $J$ because $j, k, \ell$ are distinct.

We have shown that for each edge $ij \in E(J^*)$ such that $i < j$, we have $B_{j, i} \cap B_{k}^{<} = \emptyset$. By definition of $B_{j}^{>}$, this implies that $B_{j}^{>} \cap B_{k}^{<} = \emptyset$. By \cref{claim:<disjointness}, $B_{j}^{<} \cap B_{k}^{<} = \emptyset$. Since $B_{j} = B_{j}^{<} \cup B_{j}^{>}$, we have $B_{j} \cap B_{k}^{<} = \emptyset$, as desired.
\end{proof}

As discussed before the statement of \cref{claim:first}, \cref{claim:first,claim:second} imply that item \ref{induction1} of the induction is satisfied.

By item \ref{induction2} of the induction hypothesis, $B_{1}, \dots, B_{j - 1}$ constitute branch sets of a model of $J^*[\{1, \dots, j - 1\}]$ in $H$. By \cref{claim:Bji}, for each edge $ij \in E(J^*)$ such that $i < j$, there is an edge of $H$ between $B_{j, i}$ and $B_{i}$. Therefore there is an edge of $H$ between $B_{j}$ and $B_{i}$. By \cref{claim:first}, $B_{1}, \dots, B_{j}$ constitute branch sets of a model of $J^*[\{1, \dots, j\}]$ in $H$, and item \ref{induction2} of the induction is also satisfied.

Thus all items \ref{induction1}, \ref{induction2}, \ref{induction3} of the induction are verified. For $j = |V(J^*)|$, item \ref{induction2} implies that $J^*$ is a minor of $H$, as desired.
\end{proof}

\subsection{Combining the Lemmas} \label{subsection:combination}

We now combine Lemmas \ref{lem:probabilistic} and \ref{lem:model} to prove the following reformulated version of \cref{main:rigs}.

\begin{thm} \label{thm:combination} For any integers $d, r \geqslant 0$ and real number $t \geqslant 1$, for any graph $H$ such that every minor of $H$ has edge density at most $t$, for every region intersection graph $G$ over $H$ where $G$ has maximum density at most $d$, $$\nabla_{r}(G) \leqslant et((2r + 2) d + 1).$$
\end{thm}

\begin{proof}
By \cref{thm:Hakimi}, there exists an orientation $\vec{G}$ of $G$ such that $|N_{\vec{G}}^{-}(v)| \leqslant d$ for each $v \in V(G)$. Let $J$ be a non-empty $r$-shallow minor of $G$ and let $(S_{i} : i \in V(J))$ be an $r$-shallow model of $J$ in $G$. For each $i \in V(J)$, let $c_{i}$ be the centre of $S_{i}$. For each $i \in V(J)$, let $T_{i}$ be a tree of $G[S_{i}]$ rooted at $c_{i}$ with depth at most $r$ such that $V(T_{i}) = S_{i}$. Let $\preccurlyeq$ be an arbitrary vertex ordering of $J$. Let $(Q_{ij} : ij \in E(J))$ be a corresponding collection of paths that properly represents the model $(S_{i} : i \in V(J))$. By the choice of $(T_{i} : i \in V(J))$, each path $Q_{ij}$ has at most $2r + 2$ vertices.

Let $J'$ be a non-empty junction-free subgraph of $J$. Note that $(S_{i} : i \in V(J'))$ is an $r$-shallow model of $J'$ in $G$ and $(Q_{ij} : ij \in E(J'))$ properly represents $(S_{i} : i \in V(J'))$. Let $J^*$ be the subgraph of $J'$ induced by $\{v \in V(J'): \deg_{J'}(v) \geqslant 2\}$. By \cref{lem:model} applied to $J'$, $J^*$ is a minor of $H$.

If $V(J^*) = \emptyset$ then $J'$ has maximum degree at most $1$, implying $|E(J')|/|V(J')| \leqslant \frac{1}{2} < t$. Now assume $V(J^*) \neq \emptyset$. By assumption, $|E(J^*)|/|V(J^*)| \leqslant t$. Let $A := \{v \in V(J'): \deg_{J'}(v) \leqslant 1\}$. Note that $A = V(J') \setminus V(J^*)$. By the definition of $J^*$ and since $t \geqslant 1$, $$|E(J')| \leqslant |E(J^*)| + |A| \leqslant t|V(J^*)| + |A| \leqslant t(|V(J^*)| + |A|) = t|V(J')|.$$

Thus $|E(J')|/|V(J')| \leqslant t$ regardless of whether $V(J^*) = \emptyset$. By \cref{lem:probabilistic}, $|E(J)| / |V(J)| \leqslant et((2r + 2) d + 1)$. 

Thus every $r$-shallow minor of $G$ has edge density at most $et((2r + 2) d + 1)$, as desired.
\end{proof}

\section{String Graphs} \label{section:stringgraphs}

This section presents an alternative proof that sparse string graphs have linear expansion. Here we use another measure of sparsity. A graph $G$ is \defn{$k$-degenerate} if every subgraph of $G$ has minimum degree at most $k$. The \defn{degeneracy} of $G$ is the minimum integer $k$ such that $G$ is $k$-degenerate. We implicitly use the well-known fact that $d \leqslant k\leqslant 2d$ for every graph with degeneracy $k$ and maximum density $d$.

The key observation of independent interest in our proof is that $k$-degenerate string graphs admit so-called $2k$-gap-cover-planar drawings. We start with necessary definitions. A \defn{drawing} of a graph $G$ represents each vertex of $G$ by a distinct point in a surface, and represents each edge $vw$ of $G$ by a non-self-intersecting curve between $v$ and $w$, such that no three edges cross at a single point.

The class of $k$-gap-cover-planar graphs was recently introduced by \citet{Wood25} as a combination of $k$-gap-planar graphs and $k$-cover-planar graphs. Specifically, for an integer $k \geqslant 0$, a drawing of a graph $G$ is \defn{$k$-gap-planar}~\citep{GapPlanar18} if every crossing can be charged to one of the two edges involved so that at most $k$ crossings are charged to each edge.  A graph is \defn{$k$-gap-planar} if it has a $k$-gap-planar drawing in the plane. Similar definitions were introduced by \citet{EG17} and \citet{OOW19}.

For a graph $G$ and a set $S \subseteq E(G)$, a \defn{vertex-cover} of $S$ is a set $C \subseteq V(G)$ such that every edge in $S$ has at least one endpoint in $C$. Distinct edges $e$ and $f$ in a graph are \defn{independent} if $e$ and $f$ have no common endpoint. Consider a drawing $D$ of a graph $G$. Let \defn{$D^\times$} be the set of pairs of independent edges that cross in $D$. A \defn{$D$-cover} of an edge $e\in E(G)$ is a vertex-cover of the set of edges $f$ such that $\{e,f\}\in D^\times$. Any minimal $D$-cover of $e$ does not include an endpoint of $e$. So it suffices to consider $D$-covers of an edge $vw$ in $G-v-w$. For an integer $k \geq 0$, a drawing $D$ of a graph $G$ is \defn{$k$-cover-planar} if each edge $vw\in E(G)$ has a $D$-cover in $G-v-w$ of size at most $k$. A graph $G$ is \defn{$k$-cover-planar}\footnote{There is a slight difference between this definition of $k$-cover-planar and the definition of $k$-cover-planar given by \citet{HKW}, who consider vertex-covers of the set of all edges that cross an edge $vw$ (including edges that cross $vw$ and are incident to $v$ or $w$). Every $D$-cover under this definition is a $D$-cover under our definition, and every $k$-cover-planar graph under our definition is $(k+2)$-cover-planar under the definition of \citet{HKW}, by simply adding $v$ and $w$ to the cover of each edge $vw$. We present this definition since it matches the preprint of \citet{Wood25} and the definition of gap-cover-planar graphs.}~\citep{HKW} if $G$ has a $k$-cover-planar drawing in the plane. Similar concepts were also considered by \citet*{AFPS14} and \citet*{MSSU24}.

\citet{Wood25} combined the definitions of $k$-gap-planar and $k$-cover-planar as follows. Consider a drawing $D$ of a graph $G$. A \defn{bearing} of $D$ is a set $B$ of ordered pairs $(e,f)$ with $\{e,f\}\in D^\times$, such that for each $\{e,f\}\in D^\times$ at least one of $(e,f)$ and $(f,e)$ (possibly both) is in $B$. For a bearing $B$ of $D$, a \defn{$B$-cover} of an edge $e\in E(G)$ is a vertex-cover of the set of all edges $f$ such that $(e,f)\in B$. Again, it suffices to consider $B$-covers of an edge $vw$ in $G - v - w$. For an integer $k \geq 0$, a drawing $D$ of a graph $G$ is \defn{$k$-gap-cover-planar} if there is a bearing $B$ of $D$ such that each edge $vw$ has a $B$-cover in of size at most~$k$. A graph $G$ is \defn{$k$-gap-cover-planar} if $G$ has a $k$-gap-cover-planar drawing in the plane. More generally, $G$ is \defn{$(g, k)$-gap-cover-planar} if $G$ has a $k$-gap-cover-planar drawing
in a surface of Euler genus at most $g$. The $g = 0$ case is equivalent to drawings in the plane.

\begin{lem} \label{string=gapcoverplanar} For every surface $\Sigma$ of Euler genus $g$, every $k$-degenerate string graph $G$ in $\Sigma$ has a $2k$-gap-cover-planar drawing in $\Sigma$, and thus $G$ is $(g, 2k)$-gap-cover-planar. In particular, every $k$-degenerate string graph in the plane is $2k$-gap-cover-planar.
\end{lem}

\begin{proof}

Let $n := |V(G)|$. Since $G$ is $k$-degenerate, we can enumerate the vertices $V(G) = \{v_{1}, \dots, v_{n}\}$ such that $v_{i}$ has at most $k$ neighbours in $\{v_{i + 1}, \dots, v_{n}\}$ for each $i \in \{1, \dots, n\}$. Let $N_{i} \subseteq \{v_{i + 1}, \dots, v_{n}\}$ be the set of such neighbours, so $|N_{i}| \leqslant k$.

Let $\mathcal{C} := (\gamma_{i} : i \in \{1, \dots, n\})$ be a collection of curves in $\Sigma$ that represents $G$, where each curve $\gamma_{i}$ represents~$v_{i}$. For $\varepsilon > 0$, for each $i \in \{1, \dots, n\}$, let $A_{i}^{\varepsilon} := \{p \in \Sigma : \dist_{\Sigma}(p, \gamma_{i}) < \varepsilon\}$. Choosing $\varepsilon$ to be sufficiently small, for every pair of distinct indices $i, j \in \{1, \dots, n\}$ we can assume that $A_{i}^{\varepsilon} \cap A_{j}^{\varepsilon} \neq \emptyset$ if and only if $v_{i}v_{j} \in E(G)$.

We construct a drawing of $G$ as follows. For each curve $\gamma_{i} \in \mathcal{C}$, pick an arbitrary point $x_{i} \in \gamma_{i}$ that is not an intersection point. Associate each vertex $v_{i}$ of $G$ with $x_{i}$. Associate each edge $v_{i}v_{j}$ of $G$ with a non-self-intersecting curve drawn between $x_{i}$ and $x_{j}$ in $A_{i}^{\varepsilon} \cup A_{j}^{\varepsilon}$. Draw the edges of $G$ such that no three edges cross at a single point. This constitutes a drawing $D$ of~$G$. 

We now show that $D$ is a $2k$-gap-cover-planar drawing. Let $B$ be the bearing of $D$ that consists of ordered pairs $(v_{a}v_{b}, v_{c}v_{d})$ of independent crossing edges of $G$ such that $\max(c, d) > \max(a, b)$. Consider an edge $v_{i}v_{j}$ of $G$. Let $v_{p}v_{q}$ be an edge of $G$ such that $(v_{i}v_{j}, v_{p}v_{q}) \in B$ (so $v_{p}v_{q}$ crosses $v_{i}v_{j}$). By construction, $v_{i}v_{j}$ is drawn in $A_{i}^{\varepsilon} \cup A_{j}^{\varepsilon}$, and $v_{p}v_{q}$ is drawn in $A_{p}^{\varepsilon} \cup A_{q}^{\varepsilon}$. Therefore, $v_{p} \in N_{i} \cup N_{j}$ or $v_{q} \in N_{i} \cup N_{j}$. Hence $N_{i} \cup N_{j}$ is a $B$-cover of $v_{i}v_{j}$. So each edge of $G$ has a $B$-cover of size at most $2k$. Thus $D$ is a $2k$-gap-cover-planar drawing in $\Sigma$, and $G$ is a $(g, 2k)$-gap-cover-planar graph.
\end{proof}

\citet[Corollary~9]{Wood25} proved that $\nabla_r(G) \leqslant 18(r+1)(k+1)$ for every $k$-gap-cover-planar graph $G$. So, by \cref{string=gapcoverplanar}, $\nabla_r(G) \leqslant 18(r+1)(2k+1)$ for every $k$-degenerate string graph $G$ in the plane. \citet[Theorem~21]{Wood25} proved that $\nabla_r(H) \leqslant 6(g + 3152)(r+1)(k+1)$ for every $(g, k)$-gap-cover-planar graph $H$. So, by \cref{string=gapcoverplanar}, $\nabla_r(H) \leqslant 6(g + 3152)(r+1)(2k+1)$ for every $k$-degenerate string graph $H$ in a surface with Euler genus $g$. This provides an alternative proof that sparse string graphs have linear expansion. Observe that our bounds in \cref{thm:mainplane} and \cref{thm:mainsurface} are better than these bounds. In particular, for surfaces of Euler genus $g$, our bound in \cref{thm:mainsurface} improves the above bound by the factor of $\mathcal{O}(\sqrt{g})$.

\cref{string=gapcoverplanar} shows that sparse string graphs admit gap-cover-planar drawings. We now show that both `gap' and `cover' are essential. First, the complete bipartite graph $K_{3, n}$ has $\mathcal{O}(n)$ edges and crossing number $\Omega(n^2)$ \citep{Kleitman70}. Since every $k$-gap-planar graph with $m$ edges has crossing number at most $km$ (see \citep[Property~1]{GapPlanar18}), $K_{3, n}$ is not $o(n)$-gap-planar. It is straightforward to show that $K_{3, n}$ is a $3$-degenerate string graph. Thus $3$-degenerate string graphs with $n$ vertices are not $o(n)$-gap-planar. Second, for $n \geqslant 2$ consider the graph $G_n$ obtained from the $(n \times n)$-grid by adding a dominant vertex. It is straightforward to show that $G_n$ is a $3$-degenerate string graph. The treewidth of $G_n$ is $n + 1$ (see \citep[Lemma~20]{HW17} for a proof), and hence $G_n$ has layered treewidth and row treewidth $\Omega(\sqrt{n})$. \citet{HKW} proved that every $k$-cover-planar graph with no $t$ pairwise crossing edges incident to a common vertex has layered treewidth and row treewidth at most $f(k, t)$ for some function $f$. Thus for each fixed $k$ and $t$ and for sufficiently large $n$, $G_n$ does not have a $k$-matching-planar drawing with no $t$ pairwise crossing edges incident to a common vertex. Thus sparse string graphs admit gap-cover-planar drawings, but do not admit gap-planar or cover-planar drawings.

\section{Optimality}
\label{section:optimality}

This section shows that the bound in \cref{thm:mainplane} is within an absolute constant factor of optimal. For a collection $\mathcal{C}$ of curves, the corresponding string graph is denoted by \defn{$G_{\mathcal{C}}$}. That is, $V(G_{\mathcal{C}}) = \mathcal{C}$ and $\beta_{1}\beta_{2} \in E(G_{\mathcal{C}})$ if and only if $\beta_{1}$ and $\beta_{2}$ intersect for distinct $\beta_1, \beta_2 \in \mathcal{C}$.

\begin{prop} \label{main:optimality} For any integers $d \geqslant 2$ and $r \geqslant 1$, there exists a string graph $G$ in the plane with maximum density at most $d$ such that $\nabla_{r}(G) \geqslant dr + \frac{d}{2} - r - 1$. 
\end{prop}

\begin{proof}
Let $d' : = d - 1$ and $r' := 2r + 1$. For each $j \in \{1, \dots, d'r'\}$, draw a `column' of $r'$ segments $\alpha_{1, j}, \dots, \alpha_{r', j}$ such that $G_{\{\alpha_{1, j}, \dots, \alpha_{r', j}\}}$ is a path. Draw these $d'r'$ columns so that any column is a horizontal translation of any other column and no two segments of distinct columns intersect. Here $\alpha_{i, j}$ is positioned at the $i$-th row and $j$-th column for any $i \in \{1, \dots, r'\}$ and $j \in \{1, \dots, d'r'\}$. Let $\mathcal{A} := \{\alpha_{i, j} : i \in \{1, \dots, r'\}, j \in \{1, \dots, d'r'\}\}$.

 For each $i \in \{1, \dots, r'\}$, draw $d'$ horizontal segments $\gamma_{(i, 1)}, \dots, \gamma_{(i, d')}$, each of which crosses all the segments $\alpha_{i, 1}, \dots, \alpha_{i, d'r'}$ of the $i$-th row and no other segments of~$\mathcal{A}$. Let $\Gamma := \{\gamma_{(i, k)} : i \in \{1, \dots, r'\}, k \in \{1, \dots, d'\}\}$ and $\mathcal{C} := \mathcal{A} \cup \Gamma$. Draw the segments of $\Gamma$ so that no two of them intersect each other. See \cref{figure:example}.

\begin{figure}[h]
        \centering
        \scalebox{2}{\includegraphics{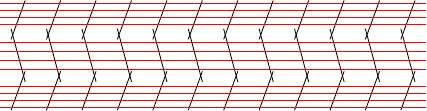}}
        \caption{Construction in the proof of \cref{main:optimality} with $r' = 3$ and $d' = 4$ (so $r = 1$ and $d = 5$). The segments of $\mathcal{A}$ are black, and the segments of $\Gamma$ are red.}
        \label{figure:example}
\end{figure}

Let $\vec{G}_{\mathcal{C}}$ be the directed graph obtained from $G_{\mathcal{C}}$ by orienting each edge $\gamma_{(i, k)}\alpha_{i, j}$ from $\gamma_{(i, k)}$ to $\alpha_{i, j}$ and each edge $\alpha_{i, j}\alpha_{i, j + 1}$ from $\alpha_{i, j}$ to $\alpha_{i, j + 1}$. By construction, $N_{\vec{G}_{\mathcal{C}}}(\beta) \leqslant d' + 1 = d$ for each $\beta \in \vec{G}_{\mathcal{C}}$. By \cref{thm:Hakimi}, $G_{\mathcal{C}}$ has maximum density at most $d$.

Let $f : \{1, \dots, d'r'\} \rightarrow \{(i, k) : i \in \{1, \dots, r'\}, k \in \{1, \dots, d'\}\}$ be a bijection. For each $j \in \{1, \dots, d'r'\}$, let $S_{j} := \{\alpha_{1, j}, \dots, \alpha_{r', j}, \gamma_{f(j)}\}$. Since $f$ is a bijection, $S_{1}, \dots, S_{d'r'}$ are pairwise disjoint. By construction, $G_{\mathcal{C}}[\{\alpha_{1, j}, \dots, \alpha_{r', j}\}]$ is a path with $r'$ vertices, and $\gamma_{f(j)}$ crosses exactly one of $\alpha_{1, j}, \dots, \alpha_{r', j}$. Therefore for each $j \in \{1, \dots, d'r'\}$, the graph $G_{\mathcal{C}}[S_{j}]$ is connected and has radius at most $\lceil \frac{r'}{2} \rceil = r$. For all distinct $j_{1}, j_{2} \in \{1, \dots, d'r'\}$, the segment $\gamma_{f(j_{1})}$, which belongs to $S_{j_{1}}$, crosses exactly one of the segments $\alpha_{1, j_2}, \dots, \alpha_{r', j_2} \in S_{j_{2}}$ of the $j_{2}$-th column. Hence, there exists an edge of $G_{\mathcal{C}}$ between $S_{j_{1}}$ and $S_{j_{2}}$. Thus $S_{1}, \dots, S_{d'r'}$ constitute branch sets of an $r$-shallow model of $K_{d'r'}$ in $G_{\mathcal{C}}$. The edge density of $K_{d'r'}$ is $\frac{d'r' - 1}{2}$, and hence $\nabla_{r}(G_{\mathcal{C}}) \geqslant \frac{d'r' - 1}{2} = \frac{(d - 1)(2r + 1) - 1}{2} = dr + \frac{d}{2} - r - 1$. This completes the proof.
\end{proof}

Since \cref{thm:mainsurface} and \cref{main:rigs} extend \cref{thm:mainplane}, \cref{main:optimality} implies that the dependence on $r$ and $d$ in these results cannot be improved to be sublinear.

\section{Colouring} \label{section:colouring}

We now explore some applications of \cref{thm:mainplane,thm:mainsurface,main:rigs} to graph colouring. \citet{KY03} introduced the following definition. For a graph $G$, total order $\preceq$ of $V(G)$, and integer $r\geq 0$, a vertex $w\in V(G)$ is \defn{$r$-reachable} from a vertex $v\in V(G)$ if there is a $vw$-path $P$ in $G$ of length at most $r$, such that $w\preceq v\prec x$ for every internal vertex $x$ in $P$. Let \defn{$\sreach_r(G,\preceq,v)$} be the set of $r$-reachable vertices from $v$. 
For a graph $G$ and integer $r\geq 0$, the (\defn{strong}) \defn{$r$-colouring number $\scol_r(G)$} is the minimum integer such that there is a total order~$\preceq$ of $V(G)$ with $|\sreach_r(G,\preceq,v)|\leq \scol_r(G)$ for every vertex $v$ of $G$. 

Generalised colouring numbers are important because they characterise bounded expansion classes \citep{Zhu09}, they characterise nowhere dense classes \citep{GKRSS18}, and have several algorithmic applications such as the constant-factor approximation algorithm for domination number by \citet{Dvorak13}, and the almost linear-time model-checking algorithm of \citet{GKS17}. See the survey by \citet{Siebertz26}. Generalised colouring numbers also provide upper bounds on several graph parameters of interest. For example, a proper vertex-colouring of a graph $G$ is \defn{acyclic} if the union of any two colour classes induces a forest; that is, every cycle is assigned at least three colours. The \defn{acyclic chromatic number} $\chi_\text{a}(G)$ of a graph $G$ is the minimum integer $k$ such that $G$ has an acyclic $k$-colouring. Acyclic colourings are qualitatively different from colourings, since every  graph with bounded acyclic chromatic number has bounded average degree. \citet{KY03} proved that every graph $G$ satisfies
\begin{equation} 
    \label{ACN} \tag{1}
    \chi_\text{a}(G)\leq \scol_2(G).
\end{equation}
Other examples include game chromatic number \citep{KT94,KY03}, Ramsey numbers \citep{CS93}, oriented chromatic number \citep{KSZ-JGT97}, arrangeability~\citep{CS93}, etc.

\citet{Zhu09} first showed that $\scol_r(G)$ is upper bounded by a function of $\nabla_r(G)$ (or more precisely, by a function of 
$\nabla_{(r-1)/2}(G)$). The best known bounds follow from results of \citet{GKRSS18}. In particular, for every integer $r\geq 1$,
\begin{equation}
    \label{scol-nabla} \tag{2}
    \scol_r(G) \leq  (6r)^r \nabla_{r-1}(G)^{3r}.
\end{equation}

First, let $t \geqslant 1$ be a real number and let $\mathcal{G}$ be a minor-closed class of graphs such that every graph in $\mathcal{G}$ has edge density at most $t$. Let $G$ be a region intersection graph over $\mathcal{G}$ such that $G$ has maximum density at most some non-negative integer $d$. By \cref{main:rigs}, $$\nabla_{r}(G) \leqslant et((2r + 2) d + 1).$$

By  \eqref{scol-nabla},   for $r\geq 1$, 
 \begin{align*}
    \scol_r(G) 
    \leq    (6r)^r (et(2rd + 1))^{3r} 
   =    6^rr^re^{3r}t^{3r}(2rd+1)^{3r}.    
\end{align*}

By \cref{ACN} with $r = 2$,
$$\chi_\text{a}(G)\leq \scol_2(G) \leq 144e^{6}t^{6}(4d + 1)^{6}.$$

Second, consider a string graph $G$ in a surface with Euler genus $g$ such that $G$ has maximum density at most some non-negative integer $d$. As explained in \cref{section:intro}, $G$ is a region intersection graph over a graph with Euler genus at most $g$, and every graph with Euler genus $g$ has maximum density at most $\sqrt{3g/2} + 3$. Hence the above inequalities for region intersection graphs can be applied, setting $t = \sqrt{3g/2} + 3$. Thus for $r \geqslant 1$, $$ \scol_r(G) \leqslant 6^rr^re^{3r}(\sqrt{3g/2} + 3)^{3r}(2rd+1)^{3r}$$ and with $r = 2$, $$\chi_\text{a}(G)\leq \scol_2(G) \leq 144e^{6}(\sqrt{3g/2} + 3)^{6}(4d + 1)^{6}.$$

Finally, consider a string graph $G$ in the plane with maximum density at most some non-negative integer $d$. Here the above inequalities can be applied, setting $g = 0$. Therefore for $r \geqslant 1$, $$ \scol_r(G) \leqslant 6^rr^re^{3r}3^{3r}(2rd+1)^{3r}$$ and with $r = 2$, $$\chi_\text{a}(G)\leq \scol_2(G) \leq 144e^{6}3^{6}(4d + 1)^{6}.$$

\subsection*{Acknowledgements}

Thanks to Jung Hon Yip for helpful comments on an early draft of this paper.

{
\fontsize{10pt}{11pt}
\selectfont
\bibliographystyle{NikolaiNatbibStyle}
\bibliography{NikolaiBibliography}
}

\end{document}